\documentclass[11pt]{amsart}
\usepackage{amsmath,amsmath,txfonts}
\pdfoutput=1 
\usepackage{hyperref}
\hypersetup{colorlinks=true, linkcolor=black, citecolor=black, filecolor=black, urlcolor=black}
\RequirePackage[scaled=0.92]{helvet}%
\RequirePackage[paperwidth=155mm,paperheight=235mm,tmargin=20mm,lmargin=19mm,rmargin=19mm,bmargin=16mm,asymmetric]{geometry}%
\makeatletter
\renewcommand\@biblabel[1]{#1.}
\makeatother

\newtheorem*{theorem*}{Theorem}

\newtheorem*{proposition*}{Proposition}

\newtheorem*{lemma*}{Lemma}

\newtheorem*{remark*}{Remark}

\newtheorem{question}{Question}
\newtheorem*{question*}{Question}

\newtheorem*{corollary*}{Corollary}

\newcommand{\Irr}{{\rm{Irr}}}
\renewcommand{\chi}{\chiup}
\renewcommand{\[}{\begin{equation*}}
\renewcommand{\]}{\end{equation*}}

\author[A. R. Miller]{Alexander R. Miller}
\title[Character zeros]
{The probability that a character value is zero for the symmetric group}
\date{\today}
\begin{document}
\begin{abstract}
We consider random character values $\chi(g)$ of the 
symmetric group $\mathfrak S_n$, where $\chi$ is chosen 
at random from the set of irreducible characters and $g$ is 
chosen at random from the group, and we show that
$\chi(g)=0$ with probability $\to 1$ as $n\to\infty$.
\end{abstract}
\maketitle
\thispagestyle{empty}
\section{Introduction}
Let $\chi$ be chosen (uniformly) at random from the irreducible 
characters of the symmetric group $\mathfrak S_n$, and let $g$ 
be chosen at random from the group itself.  What is the 
probability that $\chi(g)=0$?  In this paper we 
give a remarkable asymptotic answer of one.

\begin{theorem*}  
If $\chi$ is chosen at random from 
the irreducible characters of $\mathfrak S_n$, 
and $g$ at random from $\mathfrak S_n$, 
then $\chi(g)=0$ with probability $P_n$ $\to 1$ as $n\to\infty$.
\end{theorem*}

It will follow that the same must be true for 
the alternating group $\mathfrak A_n$.

\begin{corollary*}  
If $\chi$ is chosen at random from 
the irreducible characters of $\mathfrak A_n$, 
and $g$ at random from $\mathfrak A_n$, 
then $\chi(g)=0$ with probability $\to 1$ as $n\to\infty$.
\end{corollary*}

We prove these results in Section~\ref{Proof:Section} 
and make some remarks in Section~\ref{Remarks:Section}.

\section{Proofs}\label{Proof:Section}
Theorem will follow from the next lemma and two classical 
results about random partitions and random permutations.  
Recall that a 
partition $\lambda$ of $n$ is a set of positive integers 
$\lambda_1\geq \lambda_2\geq\ldots\geq \lambda_\ell$ 
summing to $n$.  Let $p_n$ denote the total number of 
partitions of $n$.  The cycle sizes of an element 
$g\in\mathfrak S_n$ determine a partition $\lambda$ which in 
turn determines the conjugacy class $K_\lambda$.  In particular, 
the number of conjugacy classes (resp. irreducible characters) 
of the symmetric group is equal to $p_n$.

\begin{lemma*}
Let $\Omega_n$ be a subset of the partitions of $n$.  Then
\begin{equation}
1\geq P_n\geq Q_n-|\Omega_n|/p_n,\label{Eq:Lemma}
\end{equation}
where $Q_n$ is the probability that the partition of 
an element $g\in \mathfrak S_n$ is in $\Omega_n$.
\end{lemma*}

\begin{proof}[Proof of Lemma]
Consider an element $g\in\mathfrak S_n$ whose cycle sizes 
form a partition $\lambda$ which belongs to $\Omega_n$.  Write
\[\sum\nolimits_\chi |\chi(g)|^2=z_\lambda,\]
where the sum ranges over all irreducible characters,
so that $z_\lambda$ is the size of the centralizer of $g$ 
by one of the usual orthogonality relations for characters.  
It follows that there are at least $p_n-z_\lambda$ irreducible 
characters which vanish at $g$, and thus at every conjugate 
of $g$.  Whence 
\[
P_n\geq
\frac{1}{p_nn!}\sum\nolimits_{\lambda\in\Omega_n} (p_n-z_\lambda)|K_\lambda|,
\]
where $K_\lambda$ is the class indexed by $\lambda$, so 
that $|K_\lambda|/n!=z_\lambda^{-1}$, 
and hence
\[
P_n\geq 
\sum\nolimits_{\lambda\in\Omega_n}z_\lambda^{-1}-|\Omega_n|/p_n.
\]
Now rewrite $z_\lambda^{-1}$ as $|K_\lambda|/n!$ 
to see that the sum is $Q_n$.
\end{proof}

We use equation~\eqref{Eq:Lemma} in tandem with two old results 
to show that $P_n$ tends to one by 
demonstrating a sequence of sets $\Omega_1,\Omega_2,\ldots$ such that 
$Q_n$ tends to one and $|\Omega_n|/p_n$ tends to zero.
We use a classical 
result of Erd\H{o}s and Lehner~\cite{EL} which tells us that, if 
$f(n)$ is any function which tends to infinity, then 
at most $o(p_n)$ 
(as $n\to \infty$) partitions of $n$ have a largest part $\lambda_1$ 
such that
\begin{equation}
\lambda_1\geq C\sqrt{n}(\log n+f(n)),\label{Erdos:Lehner}
\end{equation}
where $C$ is some explicit positive constant.  
We also use the following result of Goncharov~\cite{G} 
about the number of cycles $m$ of an element of $\mathfrak S_n$:
\[
{\rm{Prob.}}\left\{\alpha<\frac{m-\log n}{\sqrt{2\log n}}<\beta\right\}
\to
\pi^{-\frac{1}{2}}\! \int_\alpha^{\,\beta}\! e^{-t^2}dt,
\quad\quad n\to\infty.
\]

\begin{proof}[Proof of Theorem]
Let $\Omega_n$ be the set of partitions of $n$ which 
satisfy~\eqref{Erdos:Lehner} with $f(n)=\log n$, so that 
$|\Omega_n|/p_n$ tends to zero as $n$ tends to infinity.

To see that $Q_n$ tends to one, note that Goncharov's result 
tells us that all but at most $o(n!)$ elements of $\mathfrak S_n$ have 
$\log n+o(\log n)$ cycles, and so 
all but at most $o(n!)$ must have a cycle of size at least 
$n/(2\log n)$, which of course grows larger than $2C\sqrt{n}\log n$ as $n\to\infty$.
\end{proof}

Corollary now follows from the usual construction of the 
irreducible characters of $\mathfrak A_n$ by restricting 
down from $\mathfrak S_n$.

\begin{proof}[Proof of Corollary]
All but at most $o(|\Irr(\mathfrak A_n)|)$  
irreducible characters of $\mathfrak A_n$ 
have exactly two irreducible extensions to $\mathfrak S_n$, and 
these extensions account for all but $o(|\Irr(\mathfrak S_n)|)$ 
of the irreducible characters of $\mathfrak S_n$; indeed, it is 
well-known that two irreducible characters $\chi^\lambda,\chi^{\lambda'}$ 
of $\mathfrak S_n$ restrict to the same character if and only if 
$\lambda,\lambda'$ are conjugate, and moreover, the restriction of 
$\chi^\lambda$ is irreducible if $\lambda$ is not self-conjugate, 
and it is the sum of two distinct irreducible characters otherwise.
The result now follows from Theorem, noting that $\mathfrak A_n$ 
contains half of $\mathfrak S_n$.
\end{proof}

\section{Remarks}\label{Remarks:Section}
\subsection{}\hspace{-\parindent}%
\hspace{1.2ex}For a finite group $G$, write $P(G)$ for the probability that $\chi(g)=0$
when $\chi$ is chosen at random from the irreducible characters and $g$ 
is chosen at random from the group, so that our results translate as 
$P(\mathfrak S_n),P(\mathfrak A_n)\to 1$.  
Empirical evidence suggests that many other groups have a high proportion 
of character values equal to zero as well, and one might conjecture that 
the following question has a positive answer, perhaps even for all finite 
groups.

\begin{question}
Let $G$ be chosen uniformly at random from the set of finite simple groups 
of size less than $n$.  Then is it true that for every $\epsilon>0$ one has 
that $P(G)>1-\epsilon$ with probability $\to 1$ as $n\to\infty$?
\end{question}  

It would be interesting to show that 
$P(G)>\epsilon$ with probability $\to 1$ as $n\to\infty$ even for small 
$\epsilon$.  To this end, note that Lemma of \S2 used only general facts 
about finite groups.  The proof shows that the following is true.

\begin{proposition*}
Let $\Omega$ be a set of classes of a finite group $G$.  Then
\begin{equation}
1\geq P(G)\geq Q(G,\Omega)-R(G,\Omega),\label{PQ:Eq}
\end{equation}
where $Q(G,\Omega)$ is the proportion of $G$ covered by $\Omega$, and 
$R(G,\Omega)$ is the proportion of classes which belong to $\Omega$.  
Moreover, the right side of~\eqref{PQ:Eq} is largest when $\Omega$ is 
the set of larger than average classes.\hfill\qed
\end{proposition*}

For the last claim of the proposition, note that another way to say 
that an element belongs to a larger than average class is to say that 
$|{\rm Cl}(G)|\geq |C_G(g)|$, so that in the case of the symmetric 
group one would write that $p_n\geq z_\lambda$.

\subsection{}\hspace{-\parindent}%
\hspace{1.2ex}We also ask about choosing $\chi(g)$ at random from the character table. 

\begin{question}\label{CharTable:Question}
Let $\chi$ be chosen at random from the irreducible characters of 
$\mathfrak S_n$, and $K$ be chosen at random from the classes of 
$\mathfrak S_n$.  What can be said about the probability that 
$\chi(g_K)=0$ as $n\to \infty$?  (Here $g_K\in K$ is arbitrary.)
\end{question}

One might conjecture that the probability converges to $1/e$, or perhaps 
even $1/3$.  It would also be interesting then to investigate similar 
asymptotic questions about the nonzero entries.  For example, we ask the 
following.

\begin{question}
Does the ratio of positive to negative entries of the character table of 
$\mathfrak S_n$ tend to one as $n$ tends to infinity?
\end{question}


\end{document}